\documentclass[11pt,reqno]{amsart}
\RequirePackage[OT1]{fontenc}
\RequirePackage{amsthm,amsmath}
\RequirePackage[numbers]{natbib}
\RequirePackage[colorlinks,linkcolor=blue,citecolor=blue,urlcolor=blue]{hyperref}
\usepackage{amssymb}
\usepackage{anysize}
\usepackage{hyperref}
\usepackage{extarrows}
\usepackage{multirow}
\usepackage{natbib}
\usepackage{stmaryrd}
\usepackage{color}
\usepackage{soul}
\usepackage[disable]{todonotes}
\usepackage[perpage]{footmisc}
\usepackage{cancel}
\usepackage{amsaddr}
\usepackage{tikz}
\usetikzlibrary{patterns}

\numberwithin{equation}{section}
\newtheorem{thm}{Theorem}[section]
\newtheorem{lem}[thm]{Lemma}
\newtheorem{pro}[thm]{Proposition}
\newtheorem{cor}[thm]{Corollary}
\newtheorem{conj}[thm]{Conjecture}
\newtheorem{defi}[thm]{Definition}
\newtheorem{rem}[thm]{Remark}

\newcommand{\be}{\begin{equation}}
\newcommand{\ee}{\end{equation}}
\newcommand{\bea}{\begin{eqnarray*}}
\newcommand{\eea}{\end{eqnarray*}}

\newcommand{\eqby}[1]{\mathrel{\stackrel{#1}{=}}}

\newcommand{\leqby}[1]{\mathrel{\stackrel{#1}{\leq}}}
\newcommand{\geqby}[1]{\mathrel{\stackrel{#1}{\geq}}}
\newcommand{\deq}{\mathrel{\mathop:}=}
\newcommand\numberthis{\addtocounter{equation}{1}\tag{\theequation}} 

\makeatletter
\newcommand{\biggg}{\bBigg@\thr@@}
\newcommand{\Biggg}{\bBigg@{3.5}}

\makeatother

\makeatletter

\newcommand{\Rmnum}[1]{\expandafter\@slowromancap\romannumeral #1@}
\makeatother
\makeatletter 
\@addtoreset{equation}{section}
\makeatother  
\renewcommand\theequation{{\thesection}%
                   .{\arabic{equation}}}

\newcommand{\const}{\mbox{const}}

\newcommand{\beqa}{\begin{eqnarray}}
\newcommand{\eeqa}{\end{eqnarray}}
\newcommand{\E}{{\mathbb E }}

\newcommand{\N}{{\mathbb N}}

\newcommand{\Z}{{\mathbb Z}}
\renewcommand{\P}{{\mathbb P}}

\newcommand{\e}{\varepsilon}

\marginsize{26mm}{26mm}{26mm}{26mm}

\begin{document}
\title{Critical parameters for loop and Bernoulli percolation}
\author{Peter M\"{u}hlbacher}
\address{\textsc{University of Warwick}}
\email{\url{peter@muehlbacher.me}}
\maketitle
\begin{abstract}
	We consider a class of random loop models (including the random interchange process) that are parametrised by a time parameter $\beta\geq 0$. Intuitively, larger $\beta$ means more randomness. In particular, at $\beta=0$ we start with loops of length $1$ and as $\beta$ crosses a critical value $\beta_c$, infinite loops start to occur almost surely. Our random loop models admit a natural comparison to bond percolation with $p=1-e^{-\beta}$ on the same graph to obtain a lower bound on $\beta_c$. For those graphs of diverging vertex degree where $\beta_c$ and the critical parameter for percolation have been calculated explicitly, that inequality has been found to be an equality. In contrast, we show in this paper that for graphs of \emph{bounded degree} the inequality is strict, i.e. we show existence of an interval of values of $\beta$ where there are no infinite loops, but infinite percolation clusters almost surely.	
\end{abstract}
\section{Introduction}
The loop models considered here are percolation type probabilistic models with intimate connections to the correlation functions of certain quantum spin systems. These connections were first discovered in \cite{Toth93,AN94}. Let $G=(V,E)$ be a graph and $\beta > 0, u \in [0, 1]$ be two parameters. To each edge $e\in E$ is assigned a \emph{time} interval $[0,\beta]$, and an independent Poisson point process $X_e$ with two kinds of outcomes: ``crosses" occur with intensity $u$ and ``double bars" occur with intensity $1-u$. 

Given a realisation $(X_e)_{e\in E}$, we consider the loop passing through a point $(x, t) \in V \times [0, \beta]$ that is defined as follows (see Fig. \ref{fig:loops}). The loop is a closed trajectory with support on $V\times[0, \beta]_\text{per}$ where $[0, \beta]_\text{per}$ is the interval $[0, \beta]$ with periodic boundary conditions, i.e. the torus of length $\beta$. 
Starting at $(x, t)$, move ``up" until meeting the first cross or double bar with endpoint $x$; then jump onto the other endpoint, and continue in the same direction if a cross, in the opposite direction if a double bar; repeat until the trajectory returns to $(x, t)$. A loop is called infinite if it visits infinitely many different vertices at time $0$, and finite otherwise.\footnote{This property is more well-known in the context of random stirring models, i.e. $u=1$, first introduced by Harris \cite{Harris72}) as finite/infinite permutation \emph{cycles}. Here, however, we want to extend it to $u\neq 1$ and avoid technical difficulties encountered when defining permutations of infinite sets. We note furthermore that this unorthodox naming may be justified by the following observation (pointed out already, e.g., in \cite{HH18}): On connected graphs of bounded degree having loops visiting infinitely (or finitely) many vertices \emph{at any time} is characterised by the almost sure presence (or absence), of an ``infinite loop" as defined here.}

For graphs of sufficiently high vertex degree one expects a phase transition in the sense that there is a \emph{critical (loop) parameter} $\beta_c>0$ such that (a) for $\beta<\beta_c$ there are only finite loops and (b) for $\beta>\beta_c$ there are infinite loops almost surely. Resolving (b) is subject of ongoing research: Results have been obtained for the complete graph (\cite{Schramm,Ber10} for $u=1$, \cite{BKLM18} for $u\in[0,1]$), the hypercube (\cite{KMU16} for $u=1$), trees (\cite{Angel03,Ham12,Ham13sharp} for $u=1$, \cite{BU18loops,HH18} for $u\in[0,1]$) and the Hamming graph (\cite{MilSeng16} for $u=1$) and it remains an open problem for $G=\Z^d$ with $d\geq 2$. We can, however, easily show (a) as follows: Loop models possess a natural percolation structure when viewing any edge $e$ with $X_e$ not empty as opened; this occurs independently for all $e\in E$ with probability $1-e^{-\beta}$. We call this the percolation model with the \emph{corresponding parameter}. Since the set of vertices visited by any loop (at any time) must be contained in a single percolation cluster, only finite loops occur when percolation clusters are finite. Choosing the \emph{critical loop parameter for percolation} $\beta_c^\text{per}$ such that $1-e^{-\beta_c^\text{per}}=p_c(G,\text{bond})$, the critical parameter for bond percolation on $G$, we conclude that for $\beta<\beta_c^\text{per}$ all loops visit only finitely many vertices almost surely and hence $\beta_c \geq \beta_c^\text{per}$.

For a number of models with vertices of diverging degree it has been shown that that the above bound is in fact sharp in the sense that $\beta_c=\beta_c^\text{per}$. This is conjectured more generally for any graphs of diverging vertex degree. For $d$-regular trees it has been shown in \cite{Angel03, BU18loops} (for $u=1$ and $u\in[0,1]$) that $\beta_c$ and $\beta_c^\text{per}$ agree to first order in $d^{-1}$ as $d\to\infty$. 
Ultimately we are interested in proving statements like (a) and (b) for ``small dimensions", e.g. $G=\Z^3$, and in this case, as suggested by the asymptotic expansion around $d=\infty$ in \cite{Angel03,BU18loops} and numerics \cite{BBBUe15}, we do \emph{not} expect $\beta_c=\beta_c^\text{per}$. In fact for $u=1$ this follows for $d$-regular trees from \cite{Ham13sharp}. The contribution of this paper is to give a rigorous and robust proof of this statement for connected, countably infinite graphs $G$ of uniformly bounded degree and $u\in(0,1]$, thus extending one particular implication of \cite{Ham13sharp} to more general graphs and $u$. This leaves out the case $u=0$. We discuss in Subsection \ref{subsec:u=0} why this is hard and why new results on dependent percolation might be needed.

\begin{figure}[hbt]
  \begin{tikzpicture}[thick, scale=1.6]
    \fill[fill=gray!10] (-.5,-.2) -- (3,-.2) node[right] {$G$} -- (3.5,.2) -- (0,.2) 
      -- cycle; 
    \draw[line width=0pt, fill=yellow!30] (0,0) -- (3,0) -- (3,3) -- (0,3) node[left] {$\beta$} -- cycle; 
    
    \draw[color=black, style=thick] (0,1.85) -- (1,1.85);
    \draw[color=black, style=thick] (0,1.75) -- (1,1.75);
    \draw[color=black, style=thick] (1,.75) -- (2,.75);
    \draw[color=black, style=thick] (1,.85) -- (2,.85);
    \draw[color=black, style=thick] (1,2.4) -- (2,2.55);
    \draw[color=black, style=thick] (1,2.55) -- (2,2.4);
    \draw[color=black, style=thick] (2,1.55) -- (3,1.55);
    \draw[color=black, style=thick] (2,1.45) -- (3,1.45);
    \draw[color=black, style=thick] (2,.4) -- (3,.55);
    \draw[color=black, style=thick] (2,.55) -- (3,.4);

    \draw[color=red, line width=.5mm] (0,0) -- (0,1.75);
    \draw[color=red, line width=.5mm] (0,1.85) -- (0,3);
    \draw[color=red, line width=.5mm] (1,.85) -- (1,1.75);
    \draw[color=red, line width=.5mm] (1,1.85) -- (1,2.4);
    \draw[color=red, line width=.5mm] (2,0) -- (2,.4);
    \draw[color=red, line width=.5mm] (2,.85) -- (2,1.45);
    \draw[color=red, line width=.5mm] (2,2.55) -- (2,3);
    \draw[color=red, line width=.5mm] (3,.55) -- (3,1.45);
  
    \draw[color=blue, line width=.5mm] (1,0) -- (1,.75);
    \draw[color=blue, line width=.5mm] (1,2.55) -- (1,3);
    \draw[color=blue, line width=.5mm] (2,.55) -- (2,.75);
    \draw[color=blue, line width=.5mm] (2,1.55) -- (2,2.4);
    \draw[color=blue, line width=.5mm] (3,0) -- (3,.4);
    \draw[color=blue, line width=.5mm] (3,1.55) -- (3,3);
    
    \foreach \x in {0,...,3} 
      \draw[black,fill=white] (\x,0) circle (0.07);
  \end{tikzpicture}
  \begin{tikzpicture}[thick, scale=1.6, baseline=-9.2ex]
    \fill[fill=gray!10] (-1.3,-.8) -- (.5,-.8) -- (1.2,-.2) -- (2,-.2) node[right] {$G$} -- (2.5,.2) -- (0,.2) 
      -- cycle; 
    \draw[line width=0pt, fill=yellow!30] (-.7,-.6) -- (.3,-.6) -- (1,0) -- (2,0) -- (2,3) -- (0,3) -- (-.7,2.4) node[left] {$\beta$} 
      -- cycle; 
        
    \draw[color=gray, style=dotted] (-.7,-.6) -- (0,0) -- (1,0);
    \draw[color=gray, style=thin] (-.7,2.4) -- (.3,2.4) -- (1,3); 
    
    \draw[color=black, style=thick] (0,.55) -- (1,.7);
    \draw[color=black, style=thick] (0,.7) -- (1,.55);
    
    \draw[color=black, style=thick] (1,2.3) -- (2,2.15);
    \draw[color=black, style=thick] (1,2.15) -- (2,2.3);
    
    \draw[color=black, style=thick] (0,2.15) -- (-.7,1.4);
    \draw[color=black, style=thick] (0,2) -- (-.7,1.55);

    \draw[color=black, style=thick] (-.7,.85) -- (.3,.85);
    \draw[color=black, style=thick] (-.7,.95) -- (.3,.95);

    \draw[color=black, style=thick] (1,1.8) -- (.3,1.2);
    \draw[color=black, style=thick] (1,1.9) -- (.3,1.3);

    \draw[color=red, line width=.5mm] (0,0) -- (0,.55);
    \draw[color=blue, line width=.5mm] (0,.7) -- (0,2);
    \draw[color=red, line width=.5mm] (0,2.15) -- (0,3);

    \draw[color=blue, line width=.5mm] (1,0) -- (1,.55);
    \draw[color=red, line width=.5mm] (1,.7) -- (1,1.8);
    \draw[color=blue, line width=.5mm] (1,1.9) -- (1,2.15);
    \draw[color=blue, line width=.5mm] (1,2.3) -- (1,3);

    \draw[color=blue, line width=.5mm] (2,0) -- (2,2.15);
    \draw[color=blue, line width=.5mm] (2,2.3) -- (2,3);
    
    \draw[color=blue, line width=.5mm] (-.7,-.6) -- (-.7,.85);
    \draw[color=blue, line width=.5mm] (.3,-.6) -- (.3,.85);
    \draw[color=blue, line width=.5mm] (-.7,1.55) -- (-.7,2.4);
    \draw[color=blue, line width=.5mm] (.3,1.3) -- (.3,2.4);
    \draw[color=red, line width=.5mm] (-.7,.95) -- (-.7,1.4);
    \draw[color=red, line width=.5mm] (.3,.95) -- (.3,1.2);

    \foreach \x in {0,...,2} 
      \draw[black,fill=white] (\x,0) circle (0.07);
    \foreach \x in {-.7,.3}
      \draw[black,fill=white] (\x,-.6) circle (0.07);

  \end{tikzpicture} 

  \caption{Graphs and realizations of Poisson point processes, and their loops. In both cases, there are exactly two loops, one in red and one in blue.}\label{fig:loops}
\end{figure}
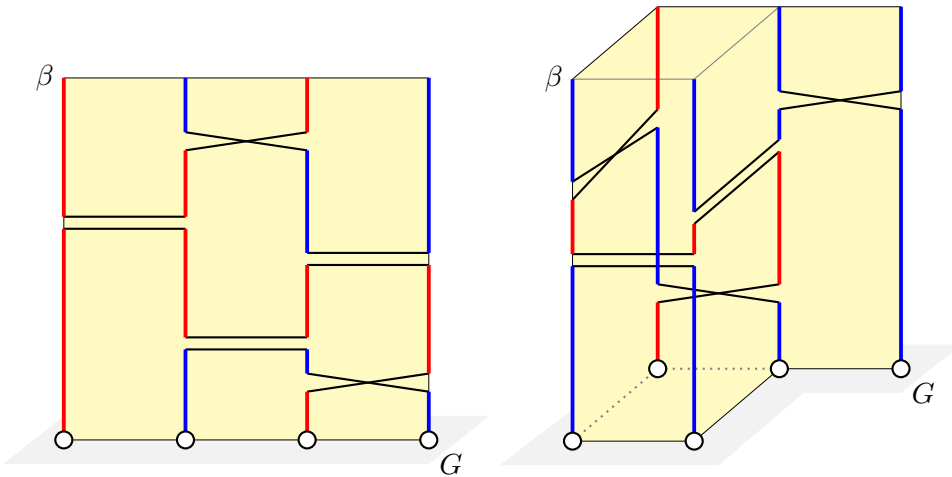

In Section \ref{sec:setting} we introduce some more notation and state our main result, Theorem \ref{thm:maincor}, rigorously. In Section \ref{sec:proof} we prove our main result. In Section \ref{sec:outlook} we discuss briefly some natural generalisations and what is expected to happen when certain assumptions of Theorem \ref{thm:maincor} are dropped. This includes a short proof of an analogous result for some expander graphs, as well as discussions how straightforward generalisations to other physically relevant models and parameter regimes are not covered by our methods or expected to fail.

\section{Setting and Result}\label{sec:setting}
Consider any countably infinite, connected, undirected graph $G=(V,E)$ of uniformly bounded degree, i.e. $\text{deg}(v)\leq\Delta$ for some $\Delta\in\N$ and all $v\in V$. From now on all graphs will be considered to be undirected, at most countably infinite and of uniformly bounded degree. In our notation we will sometimes suppress dependency on $\Delta$, which we think of as fixed from now on. We consider bond percolation on $G$. More precisely we take $\Omega\deq\{0,1\}^E$ and equip it with the $\sigma$-algebra generated by finite cylinder sets, as well as the (Bernoulli) product measure with parameter $p$: $\pi_p\deq\otimes_{e\in E}\nu_e$, where $\nu_e(1)=p=1-\nu_e(0)$, to obtain a probability space. Define its \emph{critical parameter}\footnote{In the context of sufficiently nice amenable graphs the critical value for bond percolation can also be found to be defined as the infimum over $p$ such that the expected value of $|C(v)|$ is infinite. But while these definitions agree e.g. on $\Z^d$, they may not be equivalent in general.} $p_c$ and its \emph{critical loop parameter for percolation} $\beta_c^\text{per}$ by 
\be
	p_c = p_c(G,\text{bond})\deq \inf\{p:\pi_p(|C(v)|=\infty)>0\}\quad\text{and}\quad \beta_c^\text{per}\deq -\ln(1-p_c),
\ee
where $C(v)\subseteq V$ is the set of vertices that are connected to $v$ via a path of open edges. 
Since $G$ is connected, a standard argument shows that the definition is indeed independent of the choice of $v$.

Denote the joint law of the Poisson point processes $(X_e)_{e\in E}=:X$ (henceforth also referred to as \emph{configuration}) as given in the introduction by $\P_{\beta,u}$ and let the random variable $L(v)$ be the set of vertices that the loop starting at $(v,0)\in V\times[0,\beta]$ visits at time $0$, i.e. 
\be
	L(v)=L(v)(X)\deq \{w\in V: (v,0)\leftrightarrow (w,0)\}.
\ee
Here $\{(v,t)\leftrightarrow(w,t')\}$ is the event that the two space-time points $(v,t)$ and $(w,t')$ are traversed by the same loop.
Similarly to the percolation case one can easily show, see e.g. \cite[Proposition 5]{Angel03}, that 
\be
	\beta_c=\beta_c(u)\deq \inf\{\beta:\P_{\beta,u}(|L(v)|=\infty)>0\}
\ee
is well-defined, i.e. it does not depend on the choice of $v$.

We can now state the main theorem:
\begin{thm}\label{thm:maincor}
	For all countably infinite, connected graphs $G$ of uniformly bounded degree with $p_c(G,\text{bond})<1$\footnote{This condition is a mere technicality, needed to exclude boring graphs like $G=\Z$. Note that $p_c=1$ implies $\beta_c^\text{per}=\infty$ and by the observation in the introduction ($\beta_c(u)\geq\beta_c^\text{per}$) we get the rather uninteresting equality ``$\infty=\infty$" in these cases.} and all $u\in(0,1]$ we have $\beta_c(u)>\beta_c^{per}$.
\end{thm}
The result can easily seen to hold for graphs of unbounded degree as long as their critical percolation parameter coincides with that of an induced subgraph of uniformly bounded degree.\;\todo{think about applicability to Galton-Watson trees}

To put it in the language of percolation theory, this implies that there is an interval of $\beta$ for which infinite percolation clusters for (the naturally coupled) bond percolation, but no infinite loops occur almost surely. More precisely, we obtain (as a special case) the following:
\begin{cor}
Consider $G=\Z^d, d\geq 2, u\in(0,1]$. There exist $0<\beta_1<\beta_2<\infty$ (depending only on $d, u$) such that for all $\beta\in (\beta_1,\beta_2)$ we have constants $a,b>0$ depending only on $d$ such that:
\begin{itemize}
	\item $\P_{\beta,u}(|L(0)|=k)\leq a e^{-bk}$ for all $k\in\N$ (small loops),
	\item $\pi_p(|C(0)|=\infty)>0$ for $p=1-e^{-\beta}$ (infinite percolation cluster).
\end{itemize}
\end{cor}
We stress once more that this result is not expected to be true for graphs of diverging degree; see Conjecture \ref{conj:infinitecase} for a rigorous statement. For $u=1$ it has already been ruled out for the complete graph \cite{Schramm}, the hypercube \cite{KMU16}, and the Hamming graph \cite{MilSeng16}.
\section{Proof of Theorem \ref{thm:maincor}}\label{sec:proof}

\subsection{Intuition}
Intuitively Theorem \ref{thm:maincor} tells us that for $\beta=\beta^\text{per}_c$ and all $v\in V$ we have that $L(v)$ is significantly smaller than the cluster $C(v)$ of the naturally coupled percolation process, i.e. bond percolation with $p=1-e^{-\beta}$. This is due to cancellations. Consider, for example, the empty configuration. Adding one cross at some edge $e=\{v,w\}$ will result in $\{(v,0)\leftrightarrow (w,0)\}$ (the time interval is periodic and there are no other links), but after adding a second cross on the same edge that connectivity is lost again.
The main idea of this proof is to identify and ``subtract" a local configuration that increases the likelihood of Bernoulli percolation without making the loop larger and show that it occurs sufficiently often in a sufficiently independent manner. The main challenge is the ``sufficiently independent" part, since it is not enough to show that such a local configuration occurs on a positive fraction of edges in the graph: Consider bond percolation on $\Z^2$, barely above criticality, and let us close a positive fraction of open edges. Can we conclude that there is no infinite cluster of open edges anymore? We want to draw attention to \cite{AG91} (and corrections in \cite{BBR14}) to emphasise that this is not a trivial question. (Despite being very general, \cite{AG91} does not seem to cover our case due to a lack of independence of our ``enhancements".) So unless the edges to be closed are sampled from another (independent) Bernoulli percolation measure, the answer is far from obvious. If they are, however, then it is easy to see that we just obtain another percolation process with strictly smaller parameter $p'$. This will be our goal.

\subsection{Preliminary reductions}
To make that intuition precise we need to introduce some more notation.
For all $e\in E$ and $a<b$ let $N_e(a,b]$ be the number of \emph{links} (a link is a cross or a double bar) of the configuration $X_e$ in the interval of heights $(a,b]$. For all $\beta>0$ let $n_e\deq N_e(0,\beta]$ be the total number of links on $e$ and for all edges $e$ with $n_e=2$ let $0\leq a_e<b_e\leq\beta$ denote the heights of its two links. 

Now we couple $(X_e)_{e\in E}$ to a Bernoulli bond percolation process $S=(S_e)_{e\in E}$ by colouring each edge $e\in E$ either red, blue, or uncoloured: 
\begin{enumerate}
	\item \emph{Red} ($R_e=1$) if $e$ has exactly two links ($n_e=2$), both of which are crosses and such that $N_{\tilde e}(a_e,b_e]=0$ for all $\tilde e\sim e$. 
	\item \emph{Blue} ($B_e=1$) if $e$ has at least one link ($n_e\geq 1$) and $e$ is not red. 
	\item \emph{Uncoloured}/closed ($S_e=0$) if $e$ is neither red, nor blue and\\ \emph{coloured}/open ($S_e=1$) if $e$ is red or blue (or, equivalently, if $n_e \geq 1$).
\end{enumerate}
Introduce the shorthands $R=(R_e)_{e\in E}$, $B=(B_e)_{e\in E}$, $S=(S_e)_{e\in E}$ and note that $S=B+R$. Thus $S$ is a Bernoulli bond percolation process with parameter $1-e^{-\beta}$.
It is easy to see that cycles must be subsets of percolation clusters, i.e. for any $v\in V$ we have $L(v)\subseteq C(v)$, where $C(v)$ is the vertex set of the connected component containing $v$ of the subgraph $G'$ obtained by removing all edges $e$ with $S_e=0$.
This percolation bound is too generous. Since transpositions are involutions and using commutativity of transpositions not involving the same vertices, it is again easy to see that cycles must in fact be subsets of \emph{blue} clusters, i.e.
\be\label{eq:inclusion}
	L(v)\subseteq C_B(v) (\subseteq C(v)),
\ee
where $C_B(v)$ is the set of vertices connected only by blue edges.
Now note that $\beta_c^\text{per}<\beta_c$ is equivalent to the existence of a $\beta$ such that 
\begin{itemize}
	\item[(i)] there is an infinite percolation (i.e. coloured) cluster with positive probability, but
	\item[(ii)] there is no infinite cycle a.s.
\end{itemize}
By \eqref{eq:inclusion}, (ii) follows from
\begin{itemize}
	\item[(ii')] there is no infinite blue cluster a.s.
\end{itemize}
To this end we want many edges to be red so that if we choose $\beta$ barely above criticality (for percolation) and remove all red edges, this would split up all infinite percolation clusters into finite (blue) ones.

Note that $B_e=S_e(1-R_e)$. So if $S$ and $R$ were independent Bernoulli bond percolation processes with parameters $p,p_R\in(0,1)$, respectively, then $B$ would be a Bernoulli bond percolation process with parameter $p (1-p_R)<p$. Hence choosing $p=p_c(G)+\e$ with $\e$ sufficiently small would give us a corresponding loop parameter $\beta = -\ln(1-p)$ for which there is an infinite coloured cluster with positive probability, but no infinite blue cluster a.s. 

Two problems arise: $R_e$ is not independent of $R_{e'}$, i.e. the law of $R$ is not a product measure. But it clearly suffices to show that it still dominates a non-degenerate product measure. This is the hard part and will be done in Proposition \ref{pro:main}. The other problem is that $R$ and $S$ are not independent. This can easily be overcome by noting that it suffices to consider $R$ restricted to coloured edges $E'=\{e\in E:S_e=1\}$ instead and showing that $R|_{E'}$ dominates a non-degenerate product measure on $E'$. 

Note furthermore that without loss of generality we might assume $u=1$, since considering $u\in(0,1)$ merely decreases the intensity of red edges by a factor of $u^2>0$. It will not matter if the other links are crosses or double bars, so henceforth we will suppress $u$ in the notation and talk about crosses, not links.

\begin{figure}[hbt]
\begin{tikzpicture}[thick, scale=1.6]
    \fill[fill=gray!10] (-.5,-.2) -- (3,-.2) node[right] {$G$} -- (3.5,.2) -- (0,.2) 
      -- cycle; 
    \draw[line width=0pt, fill=yellow!30] (0,0) -- (3,0) -- (3,3) -- (0,3) node[left] {$\beta$} -- cycle; 
    
    \draw[color=red, line width=.5mm] (0,0) -- (1,0);
    \draw[color=blue, line width=.5mm] (1,0) -- (3,0);

    \draw[color=black, style=thick] (0,.2) -- (1,.35);
    \draw[color=black, style=thick] (0,.35) node[left] {$a_{e_1}$} -- (1,.2);
    \draw[color=black, style=thick] (0,.6) -- (1,.75);
    \draw[color=black, style=thick] (0,.75) node[left] {$b_{e_1}$} -- (1,.6);
    \draw[color=black, style=thick] (1,1) -- (2,1);
    \draw[color=black, style=thick] (1,1.1) -- (2,1.1);
    \draw[color=black, style=thick] (1,1.4) -- (2,1.55);
    \draw[color=black, style=thick] (1,1.55) -- (2,1.4);
    \draw[color=black, style=thick] (2,2.55) -- (3,2.4);
    \draw[color=black, style=thick] (2,2.4) -- (3,2.55) node[right] {$b_{e_3}$};
    \draw[color=black, style=thick] (2,.4) -- (3,.55) node[right] {$a_{e_3}$};
    \draw[color=black, style=thick] (2,.55) -- (3,.4);
    
%
%
%
    \draw[line width=.5mm] (0,0) -- (0,.2);
    \draw[line width=.5mm] (0,.35) -- (0,.6);
    \draw[line width=.5mm] (0,.75) -- (0,3);
    
    \draw[line width=.5mm] (1,0) -- (1,.2);
    \draw[line width=.5mm] (1,.35) -- (1,.6);
    \draw[line width=.5mm] (1,.75) -- (1,1);
    \draw[line width=.5mm] (1,1.1) -- (1,1.4);
    \draw[line width=.5mm] (1,1.55) -- (1,3);

    \draw[line width=.5mm] (2,0) -- (2,.4);
    \draw[line width=.5mm] (2,2.55) -- (2,3);
    \draw[line width=.5mm] (2,.55) -- (2,1);
    \draw[line width=.5mm] (2,1.1) -- (2,1.4);
    \draw[line width=.5mm] (2,1.55) -- (2,2.4);

    \draw[line width=.5mm] (3,0) -- (3,.4);
    \draw[line width=.5mm] (3,.55) -- (3,2.4);
    \draw[line width=.5mm] (3,2.55) -- (3,3);

    \foreach \x in {0,...,3} 
      \draw[black,fill=white] (\x,0) circle (0.07);
    \draw (.5,-.1) node {$e_1$};
    \draw (1.5,-.1) node {$e_2$};
    \draw (2.5,-.1) node {$e_3$};

  \end{tikzpicture}

  \caption{An example of a configuration on three edges $e_1,e_2,e_3$. All edges are coloured ($S_{e_1}=S_{e_2}=S_{e_3}=1$) because $n_{e_1}=n_{e_2}=n_{e_3}=2>0$. $e_1$ is red ($R_{e_1}=1$) since $e_1$ has two crosses and no neighbours with a link between $a_{e_1}$ and $b_{e_1}$. $e_2$ is blue ($B_{e_2}=1$) because $n_{e_2}>0$, but not both links are crosses. $e_3$ is blue too because it has a neighbour, $e_2$, with a link between its two crosses, i.e. $N_{e_2}(a_{e_3},b_{e_3}]>0$. Note that, given that $e_1$ is red, the leftmost vertex cannot possibly be in the same loop as the others.}
\end{figure}

\begin{pro}\label{pro:main}
Fix any $\beta>0,u=1$ and some $G$ as in Theorem \ref{thm:maincor}. Let $G'=(V',E')$ be a (not necessarily connected) subgraph of $G$.
Then there exists $\delta>0$ (depending only\footnote{In particular it is independent of $G'$ and other properties of $G$.} on $\beta,\Delta$) such that the Bernoulli product measure $\pi_\delta$ is stochastically dominated by the law of $(R_e)_{e\in E'}$ conditioned on $n_e>0$ for all $e\in E'$.
\end{pro}
Henceforth we only consider the (now fixed) subgraph $G'$, so we introduce the shorthand 
\be
	\mu \deq \P_{\beta,u}(\;\cdot\;|n_e>0\;\forall e\in E'(G')),
\ee
where $\P_{\beta,u}$ is as in the introduction, but restricted to configurations on edges of $G'$.

By \cite[Lemma 1.1]{LSS97}\footnote{Intuitively one would like to apply a simple monotone coupling argument, e.g. as in \cite[p.155f]{FriedliVelenik}; a priori, however, this only works for finite graphs and for the sake of brevity we do not want to concern ourselves with unnecessarily technical limiting procedures.}, it suffices to show that
\be\label{eq:unifPos}
\inf_{e_0\in E'}\mu\left(R_{e_0}=1|R_{e_1}=\e_1,\dots,R_{e_m}=\e_m\right)\geq \delta,
\ee
uniformly over all choices of $e_i\in E'\setminus\{e_0\},\e_i\in\{0,1\}$ for $i=1,\dots,m$ and $m\in \N$, such that $\mu(R_{e_1}=\e_1,\dots,R_{e_m}=\e_m)>0$.
In what follows we consider $e_0\in E'$ to be an arbitrary fixed edge unless explicitly stated otherwise.

\begin{rem}[Notation]
	In what follows we will repeatedly abuse notation to abbreviate statements like \eqref{eq:unifPos} and its following paragraph to the shorter ``$\mu(R_{e_0}=1|\{R_{e_1},\dots,R_{e_m}\})\geq\delta$ uniformly for all $\{R_{e_1},\dots,R_{e_m}\}$". We will consider more general events than $R_{e_0}=1$ and more general \emph{conditions} than $\{R_{e_1},\dots,R_{e_m}\}$. Conditions $\{C_1=c_1,C_2=c_2,\dots,C_m=c_m\}$ will be called \emph{admissible} if there exists a configuration $X$ such that $C_1(X)=c_1, C_2(X)=c_2,\dots$. One might worry about conditioning on such admissible conditions since they may be probability zero events, but we do this only for the sake of an accessible presentation. Instead one might also discretise time (so that there is no need to condition on probability 0 events) and take limits in the end.
	
	For example $R_{e_1}=1$ and $R_{e_2}=0$ with $e_1=e_2$ is not admissible. A more interesting example of conditions that are not admissible are $X_{e_1}$ having two points at heights $a<b$, $R_{e_1}=1$ (so far it is admissible) and $N_{e_2}(a,b]>0$ for some $e_2\sim e_1$.
\end{rem}

\subsection{A spatial Markov property}
In order to avoid having to deal with complicated long-range correlations, we refine the $\sigma$-algebra by allowing to condition not only on finite collections $\{R_e\}_{e\in I\Subset E'}$, but also on some finite collection $\{X_e\}_{e\in I'\Subset E'}$ to obtain a spatial Markov property. 
For functions $f_e:X\mapsto f_e(X)$ that depend only on $X_e$ and its nearest neighbour configurations $\{X_{\tilde e}\}_{\tilde e\sim e}$ we will abuse notation and write $f_e(X_e,X_{\tilde e\sim e})$ instead of $f_e(X)$ to emphasise its dependency on only these variables. 
Endow the set of edges $E'$ with its natural metric $d=d_{E'}$, given by the graph distance of edges and define for each $A\subseteq E'$ its mollification $\overline A^{(k)}\deq \{e':\exists e\in A\text{ s.t. }d(e,e')\leq k\}$.

The following lemma is essentially a spatial Markov property (given that we have something like a ``security layer" that we have full control over):
\begin{lem}\label{lem:Markov}
	For any $A\subseteq E'$ and any collection of functions $\{f_e\}_{e\in E'}$, with $f_e$ depending only on the configuration on $e$ and its nearest neighbours $(X_{e'})_{e':e'\sim e}$, we have:
	\be\label{eq:Markov}
	\mu\left((f_{e})_{e\in A}|(X_e)_{e\in E'\setminus A},(f_e)_{e\in  E'\setminus A}\right) = \mu\left((f_{e})_{e\in A}|(X_e)_{e\in \overline A^{(2)}\setminus A},(f_e)_{e\in \overline A^{(1)} \setminus A}\right).
	\ee
\end{lem}
\begin{proof}
	Note that none of the functions $(f_e)_{e\in E'\setminus \overline A^{(1)}}$ depend on the values of $(X_e)_{e\in A}$. 
	Similarly, none of the functions that depend on $X_e$ for any $e\in A$, i.e. $(f_e)_{e\in \overline A^{(1)}}$, depend on $(X_e)_{e\in E'\setminus \overline A^{(2)}}$. Since $(X_e)_{e\in E'\setminus \overline A^{(2)}}$ and $(X_e)_{e\in \overline A^{(2)}}$ are independent under $\mu$, the result follows.
\end{proof}
	We will use Lemma \ref{lem:Markov} for $f_e = R_e$ for $e\notin A$. If for $e\in A$ we set $f_e = id$, we obtain that, conditioning as in \eqref{eq:Markov}, $X_e$ is independent of the outside, i.e. $(X_e)_{e\in E\setminus \overline A^{(2)}},(R_e)_{e\in E\setminus \overline A^{(1)}}$ or in other words: It only depends on the ``boundary" $\overline A^{(2)}\setminus A$. The same holds for $f_e = R_e$ and $f_e = n_e\deq N_e(0,\beta]$, i.e. the occurrence of double crosses and the number of crosses within $A$, respectively.

\begin{cor}
For arbitrary $A\subseteq E'$ such that $e_0\in A$, all choices of $\tilde E\subseteq E'$ with $|\tilde E|<\infty$, $e_0\notin \tilde E$ and all choices of $(R_e)_{e\in\tilde E}$ we have
	\begin{align*}
	\mu\left(R_{e_0}=1\bigg|(R_e)_{e\in\tilde E}\right)
	&\geq \inf_{(X_e)_{e\in \overline A^{(2)}\setminus A}}\mu\left(R_{e_0}=1\bigg|(R_e)_{e\in\tilde E},(X_e)_{e\in \overline A^{(2)}\setminus A}\right)\\
	&\!\!\eqby{\eqref{eq:Markov}}\!\!\! \inf_{(X_e)_{e\in \overline A^{(2)}\setminus A}}\mu\left(R_{e_0}=1\bigg|(R_e)_{e\in\tilde E\cap \overline A^{(1)}},(X_e)_{e\in \overline A^{(2)}\setminus A}\right),\numberthis\label{eq:redIneq}
	\end{align*}
	where the infimum goes over all \emph{admissible} configurations $(X_e)_{e\in \overline A^{(2)}\setminus A}$.
\end{cor}
Thus it suffices to prove that there is a $\delta>0$ such that the r.h.s. of \eqref{eq:redIneq} is bounded from below by $\delta$ for all choices of $(R_e)_{e\in \overline A^{(1)}\setminus \{e_0\}}$ for some $A$ as in the corollary.

We want to choose $A$ as small as possible. It may be tempting to choose $A=\{e_0\}$, but this does not work since one could prescribe an arbitrarily high density of crosses on one of the neighbours of $e_0$ making the probability of satisfying the conditions to colour an edge red arbitrarily small. Instead we choose $A=\{e_0\}\cup\{\tilde e\in E':\tilde e\sim e_0\}$ and prove that the probability of colouring $e_0$ red cannot be made arbitrarily small in this case.

\subsection{Pivotal edges} From now on we consider $A=\{e_0\}\cup\{\tilde e\in E':\tilde e\sim e_0\}$ fixed. Motivated by Lemma \ref{lem:Markov} we will refer to conditioning on some admissible event $\frak B=\{(R_e)_{e\in \overline A^{(1)}\setminus\{e_0\}}, (X_e)_{e\in \overline A^{(2)}\setminus A}\}$ as \emph{boundary conditions} $\frak B$. 
We formalise the intuition that we may restrict our attention to finitely many classes of configurations that ``influence" their adjacent edges' configurations. For these classes the desired properties can then be checked directly.

\begin{defi}
Given some configuration $X$, an edge $e'\in E'$, and an event (written as a function $f:X\to\{0,1\}$), we say \emph{$e'$ is pivotal for $f$}, if $f$ is not invariant under changes to $X_{e'}$ (keeping $(X_{e})_{e\neq e'}$ fixed).
\end{defi}
See Figure \ref{fig:pivotal} for an illustration.
Clearly no $e'$ outside of the support\footnote{Support as defined (e.g. in \cite[Definition 3.11]{FriedliVelenik}) for local functions; i.e. the smallest subset $\tilde E\subseteq E'$ such that if $X$ agrees with $X'$ on $\tilde E$, then $f(X)=f(X')$.} of $f$ is pivotal for $f$. The following lemma will provide more interesting examples:
\begin{lem}\label{lem:pivClassification}
	Fix any $\tilde e\sim e_0$ and let $\frak B$ be any boundary condition that does not prescribe $R_{\tilde e}$. Then $e_0$ is pivotal for $R_{\tilde e}$ if and only if $n_{\tilde e}=2$ and there is no edge $\tilde{\tilde e}\neq e_0$ adjacent to $\tilde e$ with $N_{\tilde{\tilde e}}(a_{\tilde e},b_{\tilde e}]>0$. We call this random event $P_{\tilde e}$.
\end{lem}
\begin{proof}
	Recall the condition for colouring $e$ red. The condition that $n_{\tilde e}=2$ is clear, since if that was not the case $R_{\tilde e}=0$ irrespectively of the configurations of any of its neighbours. Note furthermore that once there is a cross between the two crosses on $\tilde e$, $R_{\tilde e}=0$ irrespectively of all the other configurations.
	On the other hand, if $P_{\tilde e}$ is satisfied, one can always choose two different configurations $X_{e_0}$ such that 
	$\tilde e$ is not red with one configuration (place a cross inside $(a_{\tilde e},b_{\tilde e}]$), but red with the other (place no cross inside $(a_{\tilde e},b_{\tilde e}]$).
\end{proof}
\begin{rem}\label{rem:pivotalEffect}
	The reason we are interested in non-pivotal edges is that conditioning on $f$ such that $e$ is not pivotal for $f$ leaves the distribution of $X_e$ invariant.
	More precisely we have that the (conditional) law of $X_{e_0}$ can be written as
	\be\label{eq:pivIndep}
	\mathcal L\left(X_{e_0}|\{X_e,R_e\}_{e\in E'\setminus\{e_0\}}\right) = F\left(\{X_e,R_e\}_{e:\mathbf 1_{P_e}=1}\right),
	\ee
	with $F$ some function not depending on any other parameters (other than the fixed $\Delta,\beta$) such that $F(\emptyset)$ is $\mu$ restricted to $e_0$, i.e. the law of a Poisson point process of intensity $1$ on $[0,\beta]$, conditioned to have at least one cross.
	
	Note that in our setting $(X_e)_{e\sim e_0}$ will be random and hence the right hand side of \eqref{eq:pivIndep} ($P_e$, to be precise) will in general still depend on configurations of these non-pivotal edges.
	
\end{rem}
\begin{defi}
	Given any boundary conditions $\frak B$, we call nearest neighbours $e\sim e_0$ with $\mathbf 1_{P_e}=1$ and $R_e=0$ (or $R_e=1$) \emph{neighbours of type $0$ (or type $1$)}.
\end{defi}
\begin{figure}[hbt]

  \begin{tikzpicture}[thick, scale=1.6]
    \fill[fill=gray!10] (-.5,-.2) -- (3,-.2) node[right] {$G$} -- (3.5,.2) -- (0,.2) 
      -- cycle; 
    \draw[line width=0pt, fill=yellow!30] (0,0) -- (3,0) -- (3,3) -- (0,3) node[left] {$\beta$} -- cycle; 
    \fill[pattern=north east lines, pattern color=blue, line width=0pt] (0,0) -- (1,0) -- (1,3) -- (0,3) -- cycle;
    \draw[color=gray, style=dotted] (2,0) -- (2,3);
    \draw[color=gray, style=dotted] (1,0) -- (1,3);
    
    \draw[color=black, style=thick] (1,.75) -- (2,.9);
    \draw[color=black, style=thick] (1,.9) -- (2,.75);
    \draw[color=black, style=thick] (1,2.4) -- (2,2.55);
    \draw[color=black, style=thick] (1,2.55) -- (2,2.4);
    \draw[color=black, style=thick] (2,.4) -- (3,.55);
    \draw[color=black, style=thick] (2,.55) -- (3,.4);

    \foreach \x in {0,...,3} 
      \draw[black,fill=white] (\x,0) circle (0.07);
    \draw (.5,-.16) node {$e_0$};
    \draw (1.5,-.145) node {$\tilde e$};
    \draw (2.5,-.13) node {$\tilde{\tilde e}$};
  \end{tikzpicture}
  \begin{tikzpicture}[thick, scale=1.6]
    \fill[fill=gray!10] (-.5,-.2) -- (3,-.2) node[right] {$G$} -- (3.5,.2) -- (0,.2) 
      -- cycle; 
    \draw[line width=0pt, fill=yellow!30] (0,0) -- (3,0) -- (3,3) -- (0,3) node[left] {$\beta$} -- cycle; 
    \fill[pattern=north east lines, pattern color=blue, line width=0pt] (0,0) -- (1,0) -- (1,3) -- (0,3) -- cycle;
    \draw[color=gray, style=dotted] (2,0) -- (2,3);
    \draw[color=gray, style=dotted] (1,0) -- (1,3);
    
    \draw[color=black, style=thick] (1,.75) -- (2,.9);
    \draw[color=black, style=thick] (1,.9) -- (2,.75);
    \draw[color=black, style=thick] (1,2.4) -- (2,2.55);
    \draw[color=black, style=thick] (1,2.55) -- (2,2.4);
    \draw[color=black, style=thick] (2,1.6) -- (3,1.45);
    \draw[color=black, style=thick] (2,1.45) -- (3,1.6);
    \draw[color=black, style=thick] (2,.4) -- (3,.55);
    \draw[color=black, style=thick] (2,.55) -- (3,.4);

    \foreach \x in {0,...,3} 
      \draw[black,fill=white] (\x,0) circle (0.07);
    \draw (.5,-.16) node {$e_0$};
    \draw (1.5,-.145) node {$\tilde e$};
    \draw (2.5,-.13) node {$\tilde{\tilde e}$};
  \end{tikzpicture}

  \caption{Examples for $e_0$ being pivotal for $f=R_{\tilde e}$, i.e. $P_{\tilde e}$ (left) and $P_{\tilde e}^c$ (right). The shaded region indicates that $X_{e_0}$ is not conditioned on, i.e. is still random. Note that conditioning on $R_{\tilde e}$ being equal to $0$ (or $1$) will make $\tilde e$ a type $0$ (or a type $1$) neighbour of $e_0$ in the left picture. In the right picture conditioning on $R_{\tilde e}=1$ is not admissible.}\label{fig:pivotal}
\end{figure}
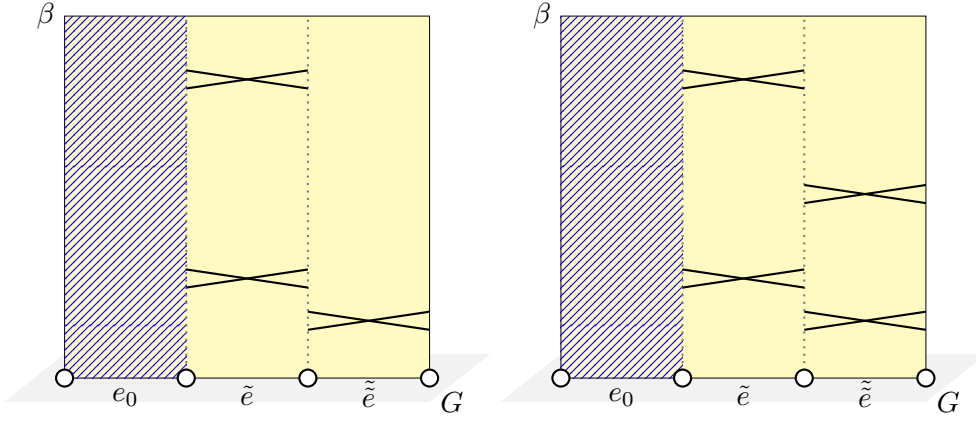

%
In the next lemma we will show that, at the price of a constant (uniform over all admissible boundary conditions), we may assume that $e_0$ has no type $0$ neighbours. 
\begin{lem}\label{lem:pivD=0}
	Fix arbitrary boundary conditions $\frak B$. Then 
	\be
		\mu\left(\bigcap_{e:R_e=0}P_e^c\bigg|\frak B\right)
		\geq \epsilon^{2(\Delta-1)}=:\eta,
	\ee
	where $\epsilon\in(0,1]$ is bounded away from $0$ uniformly over all admissible boundary conditions $\frak B$, for any fixed $\beta>0,\Delta\in\N$.
\end{lem}
\begin{proof}
	Note that it suffices to prove that
	\be
	\mu\left(P_{\tilde e}^c\big| \frak O_{\tilde e}\right)\geq\epsilon,
	\ee
	for $\epsilon>0$ as in the lemma, and some outer configuration $\frak O_{\tilde e}\deq \{(X_e,R_e)_{e\in E'\setminus\{e_0,\tilde e\}},R_{\tilde e}=0\}$ arbitrary with any $\tilde e$ such that $\tilde e\sim e_0$. 
	Denote by $K$ the event (defined on $\{n_{\tilde e}=2\}$) that there is an edge $\tilde{\tilde e}\neq e_0$ adjacent to $\tilde e$ which has a cross between the two crosses on $\tilde e$; by Lemma \ref{lem:pivClassification} we have $P_{\tilde e}^c=\{n_{\tilde e}\neq 2\}\cup (K\cap\{n_{\tilde e}=2\})$.
	Clearly 
	\be\label{eq:lnotPeId}
		\mu\left(P_{\tilde e}^c\big| \frak O_{\tilde e}\right)
		=\mu\left(K\cap \{n_{\tilde e}=2\}\big| \frak O_{\tilde e}\right) +
		  \mu\left(n_{\tilde e}\neq 2\big| \frak O_{\tilde e}\right).
	\ee
	Now assume for contradiction that for every $\e>0$ there is an outer configuration $\frak O_{\tilde e}(\e)$ such that \emph{both} terms on the right hand side of \eqref{eq:lnotPeId} are smaller than $\e/2$. In particular (negating the first term) this implies that 
	\be
		\mu\left((K^c\cap\{n_{\tilde e}=2\})\cup\{n_{\tilde e}\neq 2\}\big|\frak O_{\tilde e}(\e)\right)\geq 1-\frac{\e}{2}.
	\ee
	By assumption we have $\mu(n_{\tilde e}\neq 2|\frak O_{\tilde e}(\e))<\e/2$ and thus
	\be\label{eq:notKbig}
	\mu\left(K^c\cap\{n_{\tilde e}=2\}\big|\frak O_{\tilde e}(\e)\right)\geq 1-\e.
	\ee
	It remains to show that for $\e>0$ sufficiently small there cannot be an outer configuration $\frak O_{\tilde e}(\e)$ such that both $\mu(n_{\tilde e}\neq 2|\frak O_{\tilde e}(\e))<\e/2$ and \eqref{eq:notKbig} hold simultaneously. Note that to satisfy the former constraint for $\e$ sufficiently small we surely have to have at least one condition that changes the law of $X_{\tilde e}$. By Lemma \ref{lem:pivClassification} and Remark \ref{rem:pivotalEffect} (with $\tilde e$ taking the role of $e_0$) this can only be done by having neighbours of type $0$ or $1$. We cannot condition on neighbours being of type $0$, however, since $K$ would be true in this case and hence we could not possibly satisfy \eqref{eq:notKbig}. On the other hand it is easy to see that neighbours of type $1$ cannot possibly increase the number of crosses in probability.\footnote{Note that the result of conditioning only on type $1$ neighbours is again a Poisson point process of the same intensity (and conditioned to have at least one cross) on a strict subset of $[0,\beta]$.} Noting that under the law of a Poisson point process of intensity $1$ on $[0,\beta]$ conditioned to have at least one cross there is a strictly positive probability of having less than $2$ crosses gives the desired contradiction and thus finishes the proof.
\end{proof}
Now that we showed that with uniformly positive probability $e_0$ is not pivotal for $\tilde e\sim e_0$ with $R_{\tilde e}=0$ (and hence the law of $X_{e_0}$ is not influenced), we still need to show that for these edges $n_{\tilde e}<N$ with uniformly positive probability for some sufficiently large, but finite $N$ --- otherwise the probability of $\{R_{e_0}=1\}$ could be made arbitrarily small.
\begin{lem}\label{lem:nBnd}
	Fix arbitrary boundary conditions $\frak B$. Then
	\be
	\mu\left(\sum_{e:R_e=0\atop{e\sim e_0}}n_e<N\bigg| \bigcap_{e:R_e=0}P_e^c,\frak B\right)\geq \frac{1}{2},
	\ee
	for sufficiently large $N\in \N$ that depends only on $\beta, \Delta$.
\end{lem}
\begin{proof}
	We show a stronger statement. For any fixed edge $e'$ with arbitrary conditioning ``on the outside" $\frak O_{e'} = \{(X_e)_{e\in E\setminus\{e'\}},(R_e)_{e\in E}\}$ we have
	\begin{align*}
		\mu\left(n_{e'}>\tilde N\bigg|\bigcap_{e:R_e=0}P_e^c,\frak O_{e'}\right)\eta
		&\!\leqby{\ref{lem:pivD=0}} \mu\left(n_{e'}>\tilde N,\bigcap_{e:R_e=0} P_e^c|\frak O_{e'}\right)\\
		&\leq \frac{\E_\mu(n_{e'}|\frak O_{e'})}{\tilde N}\\
		&\leq \frac{1+|\{e:e\sim e'\}|+\beta}{\tilde N} \leq \frac{1+2(\Delta-1)+\beta}{\tilde N}.\numberthis\label{eq:nBnd}
	\end{align*}
	The third inequality follows from Lemma \ref{lem:pivClassification} and Remark \ref{rem:pivotalEffect} (with $e'$ taking the role of $e_0$). Note again that type $0$ neighbours only decrease the number of points and every type $1$ neighbour can force at most one additional cross on $e'$.\footnote{This follows from the fact that a Poisson point process on $[0,\beta]$ conditioned to have at least one cross in some interval $[a,b]\subseteq[0,\beta]$ is still just a Poisson point process on $[0,\beta]\setminus[a,b]$ and for the process on $[a,b]$ we simply note that $\P(N(a,b]>1|N(a,b]>0)=\P(N(a,b]>0)$.} Hence
	\begin{align*}
		\mu\bigg(\bigcap_{e:e\sim e_0\atop R_e=0}\{n_e\leq \tilde N\}\bigg| \bigcap_{e:R_e=0} P_e^c,\frak B\bigg) 
		\geqby{\eqref{eq:nBnd}} \left[1-\frac{1+2(\Delta-1)+\beta}{\tilde N}\eta^{-1}\right]^{2(\Delta-1)},\numberthis
	\end{align*}
	which, for $\tilde N$ sufficiently large, implies the result.
\end{proof}
Now that we can be sure there is not going to be an ``unlimited number" of crosses on edges adjacent to $e_0$ (Lemma \ref{lem:nBnd}) and that we may condition on the case that $e_0$ is only pivotal for $R_e$ that are conditioned to be $1$ (Lemma \ref{lem:pivD=0}), the last thing to prove is that there are no boundary conditions that make $\bigcup_{e:e\sim e_0\atop R_e=1} (a_e,b_e]$ arbitrarily close to $[0,\beta]$. (Recall that $0\leq a_e < b_e\leq\beta$ are the heights of the first and the second cross on the edge $e$ respectively.)
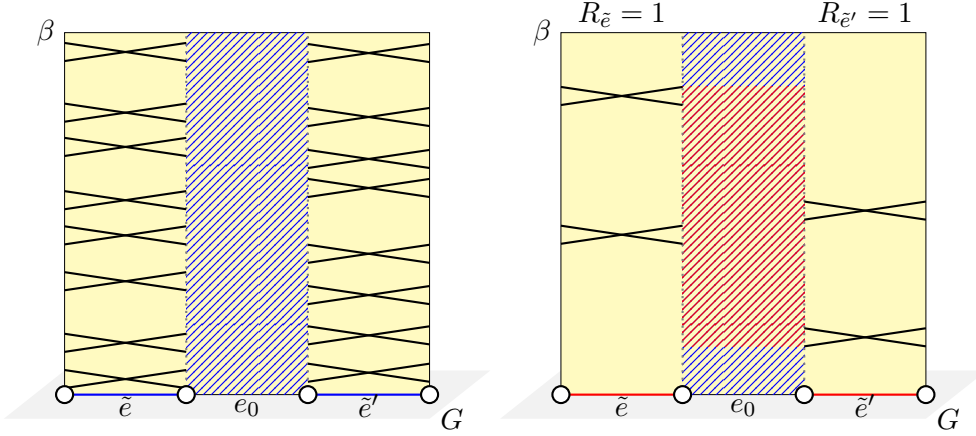
\begin{figure}[hbt]
  \begin{tikzpicture}[thick, scale=1.6]
    \fill[fill=gray!10] (-.5,-.2) -- (3,-.2) node[right] {$G$} -- (3.5,.2) -- (0,.2) 
      -- cycle; 
    \draw[line width=0pt, fill=yellow!30] (0,0) -- (3,0) -- (3,3) -- (0,3) node[left] {$\beta$} -- cycle; 
    \fill[pattern=north east lines, pattern color=blue, line width=0pt] (1,0) -- (2,0) -- (2,3) -- (1,3) -- cycle;
    \draw[color=gray, style=dotted] (2,0) -- (2,3);
    \draw[color=gray, style=dotted] (1,0) -- (1,3);
    
    \foreach \y in {0,...,7}
    {\foreach \x in {0,2}
      {
      \pgfmathsetmacro{\ycoord}{\y/2.7+.08+rand*0.1}
      \draw[color=black, style=thick] (\x,\ycoord) -- (\x+1,\ycoord+.15);
      \draw[color=black, style=thick] (\x+1,\ycoord) -- (\x,\ycoord+.15);
      }
    }    
    
    \draw[color=blue, style=thick] (0,0) -- (1,0);
    \draw[color=blue, style=thick] (2,0) -- (3,0);

    \foreach \x in {0,...,3} 
      \draw[black,fill=white] (\x,0) circle (0.07);
    \draw (.5,-.1) node {$\tilde e$};
    \draw (1.5,-.11) node {$e_0$};
    \draw (2.5,-.09) node {$\tilde e'$};
  \end{tikzpicture}
  \begin{tikzpicture}[thick, scale=1.6]
    \fill[fill=gray!10] (-.5,-.2) -- (3,-.2) node[right] {$G$} -- (3.5,.2) -- (0,.2) 
      -- cycle; 
    \draw[line width=0pt, fill=yellow!30] (0,0) -- (3,0) -- (3,3) -- (0,3) node[left] {$\beta$} -- cycle; 
    \fill[pattern=north east lines, pattern color=blue, line width=0] (1,0) -- (2,0) -- (2,3) -- (1,3) -- cycle;
    \fill[pattern=north east lines, pattern color=red] (1,.4) -- (2,.4) -- (2,2.55) -- (1,2.55) -- cycle;
    \draw[color=gray, style=dotted] (2,0) -- (2,3);
    \draw[color=gray, style=dotted] (1,0) -- (1,3);
    
    \draw[color=black, style=thick] (0,1.25) -- (1,1.4);
    \draw[color=black, style=thick] (0,1.4) -- (1,1.25);

    \draw[color=black, style=thick] (0,2.4) -- (1,2.55);
    \draw[color=black, style=thick] (0,2.55) -- (1,2.4);

    \draw[color=black, style=thick] (2,1.6) -- (3,1.45);
    \draw[color=black, style=thick] (2,1.45) -- (3,1.6);

    \draw[color=black, style=thick] (2,.4) -- (3,.55);
    \draw[color=black, style=thick] (2,.55) -- (3,.4);
   
    \draw[color=red, style=thick] (0,0) -- (1,0);
    \draw[color=red, style=thick] (2,0) -- (3,0);

    \foreach \x in {0,...,3} 
      \draw[black,fill=white] (\x,0) circle (0.07);
    \draw (.5,3.15) node {$R_{\tilde e}=1$};
    \draw (2.5,3.15) node {$R_{\tilde e'}=1$};

    \draw (.5,-.1) node {$\tilde e$};
    \draw (1.5,-.11) node {$e_0$};
    \draw (2.5,-.09) node {$\tilde e'$};
  \end{tikzpicture}

  \caption{The left picture illustrates how the probability of $e_0$ being red is small if there are many crosses on its neighbours, while the right picture illustrates what could go wrong with type $1$ neighbours. The area shaded in red emphasises that, in order not to violate $R_{\tilde e}=1$ and $R_{\tilde e'}=1$, no crosses are allowed in that region.}
\end{figure}

\begin{lem}\label{lem:D1notAll}
	Fix arbitrary boundary conditions $\frak B$. Then
	\be
	\mu\left(\bigg|[0,\beta]\setminus\bigcup_{e:e\sim e_0\atop R_e=1} (a_e,b_e]\bigg|\geq\frac{\beta}{2}\biggg|\bigcap_{e:R_e=0} P_e^c,\sum_{e:e\sim e_0\atop R_e=0}n_e<N,\frak B\right)\geq \epsilon',
	\ee
	for some sufficiently small $\epsilon'>0$ that depends only on $\beta,\Delta$. 
\end{lem}
\begin{proof}
It is enough to show that 
\be
\mu\left(b_{\tilde e}-a_{\tilde e}<\frac{\beta}{2}\frac{1}{2(\Delta-1)}\bigg|\frak O_{\tilde e}\right)\geq \epsilon''
\ee
for any red $\tilde e$ adjacent to $e_0$ and arbitrary admissible conditioning on its outside, i.e. $\frak O_{\tilde e}\deq \{(X_e)_{e\in E'\setminus\{\tilde e\}},(R_e)_{e\in E'\setminus\{\tilde e\}},R_{\tilde e}=1\}$\footnote{Note that the set of configurations where $\sum_{e:R_e=0\atop e\sim e_0}n_e<N$ is a subset of the set of admissible $\frak O_{\tilde e}$. The condition that $\bigcap_{e:R_e=0} P_e^c$ is also included. To see this, note that $P_e^c$ either because (a) $n_e\neq 2$ (in that case it is included in the $X_e$ conditioning), or (b) there is some $\tilde{\tilde e}\notin\{\tilde e,e_0\}$ s.t. $N_{\tilde{\tilde e}}(a_e,b_e]>0$ (again, this is included in the $X_{\tilde{\tilde e}}$ conditioning), or (c) $N_{\tilde e}(a_e,b_e]>0$, but this is not admissible, since $\tilde e$ is red by assumption.} and some $\epsilon''>0$ independently of $\tilde e$ and $\frak O_{\tilde e}$.

By Lemma \ref{lem:pivClassification} and Remark \ref{rem:pivotalEffect} (with $\tilde e$ taking the role of $e_0$) only type $0$ and type $1$ neighbours affect the distribution of $X_{\tilde e}$ (and hence $b_{\tilde e}-a_{\tilde e}$ which is a deterministic function of $X_{\tilde e}$). Type $0$ neighbours, however, are not admissible as they would violate $R_{\tilde e}=1$. On the other hand it is easy to see that type $1$ neighbours only make $b_{\tilde e}-a_{\tilde e}$ smaller in probability, as it results in $X_{\tilde e}$ being a Poisson point process (conditioned to have at least one cross) on a subset of $[0,\beta]$ of the same intensity.
Thus we get\todo{could make this more precise, but then it gets really unnecessarily messy; same for the proof of Lemma \ref{lem:pivD=0}} 
\be
\mu\left(b_{\tilde e}-a_{\tilde e}<\frac{\beta}{2}\frac{1}{2(\Delta-1)}\bigg|\frak O_{\tilde e}\right)\geq \mu\left(b_{\tilde e}-a_{\tilde e}<\frac{\beta}{2}\frac{1}{2(\Delta-1)}\bigg| n_{\tilde e}=2\right).
\ee
Noting that a Poisson point process of intensity $1$ on $[0,\beta]$, conditioned to have two crosses, has a strictly positive probability of placing both crosses within any given (positive) distance finishes the proof.
\end{proof}
\subsection{Reduction to a Poisson point process}
With these lemmas we can now proceed to prove Proposition \ref{pro:main} as follows:
\begin{align*}
	\mu\left(R_{e_0}=1|\frak B\right) 
	&\geq \mu\left(R_{e_0}=1,\bigcap_{e:R_e=0} P_e^c \bigg| \frak B\right)\\
	&\!\!\geqby{\ref{lem:pivD=0}}\! \mu\left(R_{e_0}=1 \bigg| \bigcap_{e:R_e=0} P_e^c,\frak B\right)\eta\\
	&\!\!\geqby{\ref{lem:nBnd}}\! \mu\left(R_{e_0}=1\bigg|\bigcap_{e:R_e=0} P_e^c,\sum_{e:e\sim e_0\atop R_e=0}n_e<N,\frak B\right)\frac{1}{2}\eta\\
	&\!\!\geqby{\ref{lem:D1notAll}}\! \mu\left(R_{e_0}=1\bigg|\bigcap_{e:R_e=0} P_e^c,\sum_{e:e\sim e_0\atop R_e=0}n_e<N,\bigg|[0,\beta]\setminus\bigcup_{e:e\sim e_0\atop R_e=1} (a_e,b_e]\bigg|\geq\frac{\beta}{2},\frak B\right)\frac{1}{2}\eta\epsilon'.\numberthis\label{eq:lastline}
\end{align*}
Now we are in a setting where $e_0$ has only ``small" type $1$ neighbours and the other neighbours have some uniformly bounded number of crosses. We proceed by remarking once again that the law of $X_{e_0}$ with only type $1$ neighbours is that of a Poisson point process conditioned to have at least one cross with intensity $1$ on $[0,\beta]\setminus(\bigcup_{e:R_e=1\atop e\sim e_0}(a_e,b_e])$ --- the Lebesgue measure of which is ensured to be larger or equal to $\beta/2$. Since the number of neighbours is bounded by $2(\Delta-1)$ it also has at most $2(\Delta-1)+1$ disconnected components. Moreover, since we do not condition on $R_{e_0}$ and we have $\bigcap_{e:R_e=0} P_e^c$, 
it follows that $(X_e)_{e\sim e_0:R_e=0}$ and $X_{e_0}$ are conditionally (conditioned on the exact configuration of type $1$ neighbours) independent. Thus (by conditioning on every admissible deterministic $(X_e)_{e\sim e_0:R_e=1}$ satisfying the bound from Lemma \ref{lem:D1notAll} and taking the essential infimum) we conclude that\todo{this may need additional explanation}
\be
\eqref{eq:lastline}\geq \P(S)\P(M)\eta\epsilon'/2=\const_{\beta,\Delta}>0,
\ee
where $S$ is the event that a Poisson point process (conditioned to have at least one cross) of intensity $1$ on $2(\Delta-1)+1$ disjoint intervals $I_i$ of size $|I_i|=\frac{\beta}{2}\frac{1}{2(\Delta-1)+1}$ drops exactly two crosses at heights $a<b$ in the same interval $I_i$ and $M$ is the event that $N$ crosses (uniformly distributed on the intervals $I_i$) do not fall between $a,b$. Clearly the product of these events is bounded away from zero. This concludes the proof of Proposition \ref{pro:main}.\qed

\section{Generalisations}\label{sec:outlook}
The result may be extended in various directions. We briefly discuss some of them.

\subsection{Expander graphs and more}\label{subsec:expanders}
Instead of one fixed infinite graph $G$, consider now a sequence of finite connected graphs $(G_n)_n = (V_n,E_n)_n$. 
A natural analogue of the critical percolation parameter $p_c$ is the smallest $p$ such that there exists a percolation cluster of size proportional to $|V_n|$ with positive probability; more formally, we define 
\be 
p_c((G_n)_n) \deq\inf\{p:\exists c>0:\lim_n\pi_{G_n,p}(|C(v)|>c|V_n|)>0\}\quad\text{and}\quad\beta_c^\text{per}((G_n)_n)\deq -\ln(1-p_c((G_n)_n)),
\ee
where $\pi_{G_n,p}$ is the Bernoulli bond-percolation measure with parameter $p$ on $G_n$. Note that without further assumptions on the sequence of graphs $(G_n)_n$ it is not clear whether $p_c((G_n)_n )$ even exists. One might think that requiring macroscopic clusters on finite subgraphs is a much stronger condition than merely infinite clusters on infinite graphs, but it is conjectured (and partially proven) that they are compatible in the sense that $\lim_n p_c((G_n)_n) = p_c(G)$, whenever $G_n$ converges in a suitable sense to $G$ and all $G_n$ are sufficiently nice -- e.g. regular expander graphs with diverging girth; see \cite{ABS04} and \cite[Conjecture 1.2]{BBLR12} for a more thorough discussion.
Similarly, we define $\beta_c$ to be 
\be
	\beta_c((G_n)_n,u)\deq \inf\{\beta:\exists c>0:\lim_n\P_{G_n,\beta,u}(|L(v)|>c|V_n|)>0)\}.
\ee

Now note that Proposition \ref{pro:main} implies that the law of the blue edges $B$ is stochastically dominated by a Bernoulli product measure with parameter $p(1-\delta)$, where $p=1-e^{-\beta}$ and $\delta>0$ depends only on $\beta$ and the maximal degree of $G$, but not on anything else like the number of vertices or its spectral gap, etc. In particular it is \emph{uniformly} bounded away from zero when considering any sequence of uniformly bounded degree graphs.

We now restrict ourselves to ``nice" families $(G_n)_n$ for which we know that $p_c((G_n)_n)$ and hence $\beta_c^\text{per}((G_n)_n)$ exist (\cite{ABS04}); in this case it follows from the above discussion:
\begin{cor}
	Let $(G_n)_n$ be a family of $d$-regular expander graphs with diverging girth. 
	For $u\in (0,1]$ we have
	\be
		\beta_c((G_n)_n,u) > \beta_c^\text{per}((G_n)_n).
	\ee
\end{cor}
In particular one observes macroscopic percolation clusters strictly before observing macroscopic loops. Depending on one's background this might not be obvious a priori since amongst bounded degree graphs expander graphs tend to behave most like the complete graph. It does not come as a surprise, however, when recalling \cite{Ham13sharp} for $u=1$ since (locally) expander graphs behave like regular trees.

\subsection{Sequences of graphs of diverging degree}
Consider, as before, sequences of finite connected graphs $(G_n)_n = (V_n,E_n)_n$; for example one might think of sequences of complete graphs or hypercubes. In order not to get unnecessarily technical, let us restrict ourselves to $G_n$ being vertex transitive for all $n$. Assume the vertex degree $\Delta_n$ of $G_n$ diverges as $n\to\infty$. 
If we chose the same definition of $p_c$ and $\beta_c$ as in Subsection \ref{subsec:expanders}, then we would get $p_c=\beta_c=0$. This prompts us to look at finer scales; it turns out that $\Theta(\Delta_n^{-1})$ is the correct one for our purposes. Instead of saying that $p_c$ and $\beta_c$ take a certain value for a sequence of graphs, one usually writes $p_c=C/\Delta_n$, so the critical parameter depends on $n$. For sequences of complete graphs, i.e. $G_n$ being the complete graph on $n+1$ vertices, one recovers the Erd\H{o}s-R\'enyi model for which it is well-known that $p_c=1/\Delta_n=1/n$.

More formally, we define $p_c,\beta_c$ to be the formal expressions 
\be
	p_c((G_n)_n)\deq \frac{t_p}{\Delta_n}
	\quad\text{and}\quad 
	\beta_c((G_n)_n,u)\deq \frac{t_\beta}{\Delta_n},
\ee
where (for a fixed sequence $(G_n)_n$ and fixed $u$) we define
\begin{align}
	t_p &\deq \inf\{t:\exists c>0:\lim_n\pi_{G_n,p=t/\Delta_n}(|C(v)|>c|V_n|)>0\},\\
	t_\beta &\deq\inf\{t:\exists c>0:\lim_n\P_{G_n,\beta=t/\Delta_n,u}(|L(v)|>c|V_n|>0)\},
\end{align}
whenever the limits are well-defined. Just as before define $\beta_c^\text{per}((G_n)_n)\deq -\ln(1-p_c((G_n)_n))$. We can now state the conjecture complementing our main result, Theorem \ref{thm:maincor}. This addresses the question what is expected to happen if the assumption of having uniformly bounded vertex degree is dropped:
\begin{conj}\label{conj:infinitecase}
	Consider a sequence of finite, connected, vertex transitive\footnote{We do not expect this to be a necessary requirement, but it allows us to easily make sense of ``diverging vertex degree". While this conjecture might extend to non-vertex transitive graphs where the vertex degree for all vertices is bounded from below by some diverging sequence $a_n\to\infty$, we do not make any claims about more pathological graphs, where e.g. only half the vertices are of bounded degree and the other half diverges quickly.} graphs $(G_n)_n$ with diverging vertex degree $\Delta_n$ such that $p_c((G_n)_n)$ and $\beta_c((G_n)_n,u)$ exist. Then for $u\in[0,1]$ we expect
	\be\beta_c((G_n)_n,u) = \beta_c^\text{per}((G_n)_n).\ee
\end{conj}


\subsection{Weighing configurations by $\theta^{\#\text{loops}}$}
One natural generalisation is a weighted version of the loop model presented in the introduction. T\'{o}th showed that a variant of this model, where permutations receive the weight $\theta^{\#\text{cycles}}$ ($\theta=2$), is closely related to the quantum Heisenberg ferromagnet \cite{Toth93}. Another loop model was introduced by Aizenman and Nachtergaele to describe spin correlations of the quantum Heisenberg antiferromagnet \cite{AN94}. These loop models were combined in order to describe a family of quantum systems that interpolate between the two Heisenberg models, and which contains the quantum XY model \cite{Ue13}. In order to represent a quantum model, one should choose the weight $\theta = 2, 3, 4, \dots$; quantum correlations are then given in terms of loop correlations, and magnetic long-range order is equivalent to the presence of macroscopic loops. Notice that the parameter $\beta$ plays the role of the inverse temperature of the quantum spin system, hence the notation.

More specifically, instead of $\P_{\beta,u}$ consider $\P_{\beta,u,\theta}$, which is characterised by its Radon-Nikodym derivative with respect to $\P_{\beta,u}$ as follows: 
\be
\frac{d\P_{\beta,u,\theta}}{d\P_{\beta,u}}(X) = \frac{\theta^{\ell(X)}}{Z_{\beta,u,\theta}}.
\ee
Here $Z_{\beta,u,\theta}$ is the appropriate normalisation and $\ell(X)$ denotes the number of loops in the configuration $X$. A priori this is only well-defined on finite $G$.

Note that for $\theta\neq 1$ we introduce some global dependencies in the sense that, even though the reference measure $\P_{\beta,u}$ was a product measure on edges $e\in E$, the configurations $X_e$ and $X_{e'}$ are not independent under $\P_{\beta,u,\theta}$ anymore. This creates additional complications which have been addressed in recent works such as \cite{B15largecycles, B16weighted,AlonKozma18,BFU18} on the complete graph, \cite{BU18,BEL18} on trees and \cite{AKM18} on the Hamming graph.

What process is the loop percolation process naturally compared to? By the lack of independence of configurations on different edges, Bernoulli percolation gives a crude bound, but it easy to convince oneself that this is not sharp. In \cite{BBBUe15} it is briefly discussed why the random cluster model, closely related to the Potts and Ising model, is a potential candidate for a natural comparison giving sharp bounds in some cases and numerical estimates are given.

We conclude this subsection by noting that our results (Proposition \ref{pro:main}) are robust enough to carry over to the $\theta > 1$ case in an appropriate sense, but we lack another model (like percolation for $\theta=1$) to naturally compare it to; hence there is no obvious meaningful generalisation of Theorem \ref{thm:maincor} to this physically interesting regime.

\subsection{The $u=0$ case}\label{subsec:u=0} Numerics \cite{BBBUe15} suggest that the result should be true in this case as well. However, a few new ideas seem to be needed since two subsequent double bars do not cancel each other out. More generally, it is not hard to see that there cannot be any other analogue\footnote{More precisely we mean any function $H:X\to\{0,1\}^E$ with (a) $H_e(X)=1$ only for $X$ such that there is a non-empty configuration $X_e$ on $e$, but in terms of loops it would not make a difference if we replaced $X_e$ by the empty configuration and (b) $H=(H_e)_{e\in E}$, as a stochastic process, can be bounded (stochastically) from below by some Bernoulli percolation measure $\pi_\delta$ of strictly positive parameter $\delta>0$.} of $R_e$ satisfying \eqref{eq:unifPos} on finite graphs: For any configuration enumerate all the vertices of $G$ in an arbitrary way, pick the first one, say $x_1$, and mark it with a $+$; follow the loop emanating from $(x_1,0)$ in the upward direction and mark every vertex it visits at time $0$ with a $+$ or $-$, depending on whether we traverse it in upward or downward direction, respectively; remove marked vertices from the list, pick the next one from the list and continue this process until no vertex is left in the list. It is immediate that unless all vertices are marked with $+$'s the configuration cannot correspond to the identity permutation on the vertices of $G$, which would be obtained if said analogue of $R_e$ was equal to $1$ for all $e\in E$. Noting that there is no configuration with at least one double bar for which the above algorithm does not assign a ``$-$" to at least one vertex (hence it cannot be equal to the configuration where no double bars have been added) concludes the argument.

It appears that this is not merely a technicality. This is because even after restricting our attention to infinite graphs, there does not seem to be an obvious analogue of $R_e$ to satisfy \eqref{eq:unifPos} (essentially by a quantified version of the above argument when conditioning on $X_e$ on some boundary).

\section{Acknowledgements}
I thank Daniel Ueltschi and Eleanor Archer for many useful discussions.

\bibliography{references}{}

\newcommand{\etalchar}[1]{$^{#1}$}
\begin{thebibliography}{BKLM18}

\bibitem[ABS{\etalchar{+}}04]{ABS04}
Noga Alon, Itai Benjamini, Alan Stacey, et~al.
\newblock Percolation on finite graphs and isoperimetric inequalities.
\newblock {\em The Annals of Probability}, 32(3):1727--1745, 2004.

\bibitem[AG91]{AG91}
Michael Aizenman and Geoffrey Grimmett.
\newblock Strict monotonicity for critical points in percolation and
  ferromagnetic models.
\newblock {\em Journal of Statistical Physics}, 63(5-6):817--835, 1991.

\bibitem[AK18]{AlonKozma18}
Gil Alon and Gady Kozma.
\newblock The mean-field quantum heisenberg ferromagnet via representation
  theory.
\newblock {\em arXiv preprint arXiv:1811.10530}, 2018.

\bibitem[AKM18]{AKM18}
Rados{\l}aw Adamczak, Micha{\l} Kotowski, and Piotr Mi{\l}o{\'s}.
\newblock Phase transition for the interchange and quantum heisenberg models on
  the hamming graph.
\newblock {\em arXiv preprint arXiv:1808.08902}, 2018.

\bibitem[AN94]{AN94}
Michael Aizenman and Bruno Nachtergaele.
\newblock Geometric aspects of quantum spin states.
\newblock {\em Communications in Mathematical Physics}, 164(1):17--63, 1994.

\bibitem[Ang03]{Angel03}
Omer Angel.
\newblock Random infinite permutations and the cyclic time random walk.
\newblock {\em Discrete Mathematics and Theoretical Computer Science AC}, pages
  9--16, 2003.

\bibitem[BBBU15]{BBBUe15}
Alessandro Barp, Edoardo~Gabriele Barp, Fran{\c c}ois-Xavier Briol, and Daniel
  Ueltschi.
\newblock A numerical study of the 3d random interchange and random loop
  models.
\newblock {\em Journal of Physics A: Mathematical and Theoretical},
  48(34):345002, 2015.

\bibitem[BBL{\etalchar{+}}12]{BBLR12}
Itai Benjamini, St{\'e}phane Boucheron, G{\'a}bor Lugosi, Rapha{\"e}l
  Rossignol, et~al.
\newblock Sharp threshold for percolation on expanders.
\newblock {\em The Annals of Probability}, 40(1):130--145, 2012.

\bibitem[BBR14]{BBR14}
Paul Balister, B{\'e}la Bollob{\'a}s, and Oliver Riordan.
\newblock Essential enhancements revisited.
\newblock {\em arXiv preprint arXiv:1402.0834}, 2014.

\bibitem[BEL18]{BEL18}
Volker Betz, Johannes Ehlert, and Benjamin Lees.
\newblock Phase transition for loop representations of quantum spin systems on
  trees.
\newblock {\em arXiv preprint arXiv:1804.00860}, 2018.

\bibitem[Ber10]{Ber10}
Nathana{\"e}l Berestycki.
\newblock Emergence of giant cycles and slowdown transition in random
  transpositions and $k$-cycles.
\newblock {\em Electronic Journal of Probability}, 16:152--173, 2010.

\bibitem[BFU18]{BFU18}
Jakob Bj{\"o}rnberg, J{\"u}rg Fr{\"o}hlich, and Daniel Ueltschi.
\newblock Quantum spins and random loops on the complete graph.
\newblock {\em arXiv preprint arXiv:1811.12834}, 2018.

\bibitem[Bj{\"o}15]{B15largecycles}
Jakob Bj{\"o}rnberg.
\newblock Large cycles in random permutations related to the heisenberg model.
\newblock {\em Electronic Communications in Probability}, 20, 2015.

\bibitem[Bj{\"o}16]{B16weighted}
Jakob Bj{\"o}rnberg.
\newblock The free energy in a class of quantum spin systems and interchange
  processes.
\newblock {\em Journal of Mathematical Physics}, 57(7):073303, 2016.

\bibitem[BKLM18]{BKLM18}
Jakob Bj{\"o}rnberg, Micha{\l} Kotowski, Benjamin Lees, and Piotr Mi{\l}o{\'s}.
\newblock The interchange process with reversals on the complete graph.
\newblock {\em arXiv preprint arXiv:1812.03301}, 2018.

\bibitem[BU18a]{BU18loops}
Jakob Bj{\"o}rnberg and Daniel Ueltschi.
\newblock Critical parameter of random loop model on trees.
\newblock {\em The Annals of Applied Probability}, 28(4):2063--2082, 2018.

\bibitem[BU18b]{BU18}
Jakob Bj{\"o}rnberg and Daniel Ueltschi.
\newblock Critical temperature of heisenberg models on regular trees, via
  random loops.
\newblock {\em arXiv preprint arXiv:1803.11430}, 2018.

\bibitem[FV17]{FriedliVelenik}
Sacha Friedli and Yvan Velenik.
\newblock {\em Statistical Mechanics of Lattice Systems: A Concrete
  Mathematical Introduction}.
\newblock Cambridge University Press, 2017.

\bibitem[Ham12]{Ham12}
Alan Hammond.
\newblock Infinite cycles in the random stirring model on trees.
\newblock {\em Bulletin of the Institute of Mathematics Academia Sinica},
  8:85--104, 2012.

\bibitem[Ham15]{Ham13sharp}
Alan Hammond.
\newblock Sharp phase transition in the random stirring model on trees.
\newblock {\em Probability Theory and Related Fields}, 161(3-4):429--448, 2015.

\bibitem[Har72]{Harris72}
Theodore~E. Harris.
\newblock Nearest-neighbor markov interaction processes on multidimensional
  lattices.
\newblock {\em Advances in Mathematics}, 9(1):66--89, 1972.

\bibitem[HH18]{HH18}
Alan Hammond and Milind Hegde.
\newblock Critical point for infinite cycles in a random loop model on trees.
\newblock {\em arXiv preprint arXiv:1805.11772}, 2018.

\bibitem[KMU16]{KMU16}
Roman Koteck{\'y}, Piotr Mi{\l}o{\'s}, and Daniel Ueltschi.
\newblock The random interchange process on the hypercube.
\newblock {\em Electronic Communications in Probability}, 21:1--9, 2016.

\bibitem[LSS97]{LSS97}
Thomas~M Liggett, Roberto~H Schonmann, and Alan~M Stacey.
\newblock Domination by product measures.
\newblock {\em The Annals of Probability}, 25(1):71--95, 1997.

\bibitem[M{\c S}16]{MilSeng16}
Piotr Mi{\l}o{\'s} and Bati {\c S}eng{\"u}l.
\newblock Existence of a phase transition of the interchange process on the
  hamming graph.
\newblock {\em arXiv preprint arXiv:1605.03548}, 2016.

\bibitem[Sch05]{Schramm}
Oded Schramm.
\newblock Compositions of random transpositions.
\newblock {\em Israel Journal of Mathematics}, 147(1):221--243, 2005.

\bibitem[T{\'o}t93]{Toth93}
B{\'a}lint T{\'o}th.
\newblock Improved lower bound on the thermodynamic pressure of the spin 1/2
  heisenberg ferromagnet.
\newblock {\em Letters in Mathematical Physics}, 28(1):75--84, 1993.

\bibitem[Uel13]{Ue13}
Daniel Ueltschi.
\newblock Random loop representations for quantum spin systems.
\newblock {\em Journal of Mathematical Physics}, 54(8), 2013.

\end{thebibliography}
\bibliographystyle{alpha}

\end{document}